\documentclass[12pt]{amsart}
\usepackage{amssymb,amscd}

\newtheorem{thm}{Theorem}[section]
\newtheorem{prop}[thm]{Proposition}
\newtheorem{lem}[thm]{Lemma}
\newtheorem{rem}[thm]{Remark}

\newtheorem{definition}[thm]{Definition} 
\numberwithin{equation}{section}

\newcommand{\Q}{\mathbb{Q}}
\newcommand{\Z}{\mathbb{Z}}
\newcommand{\C}{\mathbb{C}}
\newcommand{\F}{\mathbb{F}}
\newcommand{\T}{\mathfrak{t}}

\begin{document}

\title{ A  description based on Schubert classes  \\
      of   \\
   cohomology  of  flag manifolds   }
\author{Masaki Nakagawa$^*$ }

\pagestyle{plain}



\thanks{
$^*$ Partially supported by the Grant-in-Aid for Scientific  Research 
(C) 18540106, Japan Society for  the Promotion of Science. 
}

\address{Department of General Education \endgraf
                      Takamatsu National College of Technology \endgraf
                      Takamatsu 761-8058 \\ Japan}
\email{nakagawa@takamatsu-nct.ac.jp}

\maketitle

\begin{abstract}
We  describe  the integral cohomology   rings of the  flag manifolds of types  $B_{n}, D_{n}, G_{2}$
and  $F_{4}$ in terms of their  Schubert classes.  The main tool is the divided difference operators
of  Bernstein-Gelfand-Gelfand and  Demazure.  As an application, we compute the Chow rings of  
the corresponding complex algebraic groups,  recovering thereby  the results of R. Marlin. 
 \end{abstract} 

\section{Introduction} 
Let $K$ be a compact connected Lie group and $T$ its  maximal torus. The homogeneous space 
$K/T$, called the  flag manifold,  plays an important role in algebraic topology, 
algebraic geometry and representation theory. 

In this paper, we are concerned with the integral cohomology  of  the flag 
manifold $K/T$. As is well known,   there are two descriptions of the integral cohomology  of 
flag manifold. The first one is the ``{\it Borel presentation},"   due to  A. Borel \cite{Bor1}, 
which  identifies  the rational cohomology ring of $K/T$ with the quotient ring of a  polynomial 
ring by  its  ideal generated by $W$-invariants  of positive degrees, where $W$ is the Weyl 
group of $K$. Combining  Borel's result and the known structures of the mod $p$ 
cohomology rings of $K$, H. Toda gave  general descriptions of the integral cohomology 
rings of $K/T$ for all $K$ simple \cite{Toda1}.  So far  the integral cohomology rings 
of flag manifolds  for all compact simply connected  simple Lie groups  are  determined
(see \cite{Bor1}, \cite{Bott-Sam1},  \cite{Toda-Wat1}, \cite{Nak1}, \cite{Nak2}).
The second  is the ``{\it Schubert presentation}" which describes the integral cohomology  $H^*(K/T;\Z)$
in terms of the Schubert classes corresponding to the Schubert varieties  derived  from  the Bruhat 
decomposition of $G = K^{\C}$, the complexification of $K$. 

In the Borel presentation, which is given  by generators and relations, 
the ring structure of $H^*(K/T;\Z)$ can be relatively easy to  obtain. 
However,  the generators in this  presentation  have little geometric meaning.  In contrast, 
in the Schubert presentation,  the Schubert classes  correspond  to the geometric objects 
- the Schubert varieties - and they  form an additive basis for  $H^*(K/T;\Z)$ (\cite{Che1}). 
As a disadvantage,  the  multiplicative  structure among them  is closely related to  the intersection 
multiplicities, and is highly complicated in general.

Up to now, there have been several attempts to establish  a connection between  the two descriptions 
for some types of spaces (see, e.g.,  \cite{BGG1}, \cite{DZ1},  \cite{Ili-Man1}).  
The main  aim  of this paper also falls in this category.  More precisely, we express the ring 
generators in the Borel presentation of  $H^*(K/T;\Z)$ for $K = SO(n), G_2$ and $F_{4}$ 
in terms of Schubert classes.     For this, we make use of the ``{\it divided difference 
operators}" introduced independently by   Bernstein-Gelfand-Gelfand  \cite{BGG1} and 
Demazure  \cite{Dem1}.   For $K$ as above,  there exist  extra generators of degrees greater
than two.  So  we cannot  apply  the divided difference operators  directly to these  higher
generators.  Fortunately, using  the classical fact that,  rationally, the cohomology of $K/T$ is 
generated as a ring by two dimensional elements  and the integral cohomology of $K/T$ has no torsion, 
we can carry out the computation.      
An additional aim of this  paper is to apply our results to recovering  the Chow rings of the complex 
algebraic groups $\mathrm{SO}(n), \mathrm{Spin}(n), \mathrm{G}_{2}$  and $\mathrm{F}_{4}$,  which  
were originally computed by  R. Marlin  \cite{Mar1}.
 (In this paper,  we denote the compact Lie groups, e.g., by $SO(n), Spin(n), G_2, F_4$, while their  
complexifications by $\mathrm{SO}(n), \mathrm{Spin}(n), \mathrm{G}_{2}, \mathrm{F}_{4}$ 
respectively.)   
In order to determine  the Chow rings of the corresponding flag manifolds, Marlin relied on the result
of Demazure  \cite{Dem1} which describes them as the ``{\it cohomology rings of the root system},"  and 
he made elaborate  computations.   In this paper, we simplify  Marlin's  computations, using  the Borel 
presentation of $H^*(K/T;\Z)$ and our  result  mentioned above.      
          
The paper is organized as follows.  In Section 2, we briefly review the cohomology of flag 
manifolds,  emphasizing the difference between the Borel presentation and the Schubert 
presentation.  In Section 3, we introduce the divided difference operators of  
Bernstein-Gelfand-Gelfand  and Demazure and collect the results used later. Sections 4 to 6 
are  devoted to  computations,   and we obtain there the  main results of this paper
(see  Propositions  \ref{prop:1st_main1}, \ref{prop:1st_main2}, \ref{prop:1st_main3} and 
\ref{prop:1st_main4}). 
In Section 7, following  Grothendieck's remark (\cite{Gro1}, p.21,  R{\scriptsize EMARQUES} $2^{\circ}$), 
we compute the Chow rings of $\mathrm{SO}(n), \mathrm{Spin}(n), \mathrm{G}_{2}$ and $\mathrm{F}_{4}$ 
(see  Theorems \ref{thm:2nd.main1}, \ref{thm:2nd.main2}, \ref{thm:2nd.main3} and \ref{thm:2nd.main4}).  
    
We observe that the method  of this paper can also be applied to the exceptional Lie groups 
$E_{6}, E_{7}$ and $E_{8}$. Indeed, we succeeded in computing  the Chow rings of the complex 
algebraic groups ${\rm E_6}$ and ${\rm E_7}$ in  \cite{Kaji-Nak1}.

{\it  Acknowledgments.}  
Firstly, we  thank  J\'{u}lius Korba\v{s} for reading carefully  the manuscript and 
giving us  useful comments and William  Fulton  for kindly pointing  out an error concerning the 
Schubert classes in the earlier version of the manuscript.    
Secondly, we  thank  Mamoru Mimura  for  giving us various suggestions. 
Finally, we thank  Shizuo Kaji for the  program using Maple,  which confirmed our computations.

\section{The cohomology  of flag manifolds}
In this section we briefly review the Borel presentation and the Schubert presentation of the 
cohomology of flag manifolds.   

We introduce the notation that is needed in the sequel.  
\begin{enumerate} 
  \item []  $K$:  a compact simply connected simple Lie group of rank $l$;   
  \item []  $T$:  a maximal  torus of $K$;  
  \item []  $G = K^{\C}$: a complexification of $K$;  
  \item []  $B$: a Borel subgroup containing $T$; 
  \item []  $\T$: the Lie algebra of $T$, $\T^{*}$:  the dual space of $\T$;
  \item []  $(\,\cdot \, | \, \cdot \, )$: the  invariant inner product  
            on $\T$ (or on $\T^*$);
  \item []  $\Delta \subset \T^{*}$: the root system with respect to $T$;   
  \item []  $\Delta^{+}$: a set of positive roots, $\Delta^{-} = - \Delta^{+}$; 
  \item []  $\Pi = \{ \alpha_{1}, \ldots, \alpha_{l} \}$: the system of  simple roots;
  \item []  $\alpha^{\vee} = \dfrac{2\alpha}{(\alpha|\alpha)}$: the coroot corresponding to $\alpha \in \Delta$;
  \item []  $\omega_{i} \; (1 \leq i \leq l)$: the $i$th fundamental weight, satisfying 
            $(\omega_{i}|\alpha_{j}^{\vee}) = \delta_{ij}$;   
  \item []  $s_{i} = s_{\alpha_{i}} \; (1 \leq i \leq l)$: the  reflection corresponding to 
            the simple root $\alpha_{i}$;
  \item []  $W = W(K)$: the Weyl group of $K$ generated by the simple reflections 
            $S = \{ s_{1}, \ldots, s_{l} \}$;  
  \item []  $l(w)$: the length of an element $w \in W$ with respect to $\{ s_{1}, \ldots, s_{l} \}$; 
  \item []  $w_{0}$: the longest element of $W$; 
  \item []  $e_{i}(x_{1}, \ldots, x_{n})$: the $i$th elementary symmetric function in variables 
            $x_{1}, \ldots, x_{n}$.  
\end{enumerate} 

Now  we review  the Borel presentation. The inclusion $T \hookrightarrow K$ induces the  classical fibration 
  \begin{equation*} 
           K/T \overset{\iota}{\longrightarrow} BT \overset{\rho}{\longrightarrow} BK,    
   \end{equation*} 
where $BT$ (resp. $BK$) denotes the classifying space of $T$ (resp. $K$).
The induced homomorphism 
 \begin{equation}  \label{eqn:ch.hom} 
     c = \iota^*: H^{*}(BT;\Z) \longrightarrow H^*(K/T;\Z)      
  \end{equation} 
is called the {\it characteristic homomorphism} and  plays a crucial role in Borel's work.  
The Weyl group $W$ acts naturally on $T$,  hence  on $H^2(BT;\Z)$.  
We extend this natural action of $W$ to the whole  $H^*(BT;\Z)$ and also to  
$H^*(BT;\F) = H^*(BT;\Z) \otimes_{\Z} \F$, where $\F$ is any field. We denote by $H^*(BT;\Z)^{W}$
(resp. $H^*(BT;\F)^{W}$)  the ring of  $W$-invariants in $H^*(BT;\Z)$ (resp. $H^*(BT;\F)$).     
Then one of the main results of  Borel can be  stated as follows. 
     \begin{thm}[Borel  \cite{Bor1}] 
       Let $\F$ be a field of characteristic zero. Then the characteristic homomorphism
       induces an  isomorphism 
            \begin{equation*} 
                 \bar{c} : H^{*}(BT;\F)/(H^{+}(BT;\F)^{W})  \longrightarrow   H^{*}(K/T;\F),  
             \end{equation*} 
         where $(H^{+}(BT;\F)^{W})$ is the ideal in $H^*(BT;\F)$ generated by the   
          $W$-invariants of positive degrees. 
      \end{thm} 
In particular,  one can reduce  the computation of  the rational cohomology ring $H^*(K/T;\Q)$ 
to that  of the ring of invariants $H^{*}(BT;\Q)^{W}$.  Observe that  $H^*(K/T;\Q)$  is generated by 
$H^2(K/T;\Q)$ as a ring.  
In order to determine the integral cohomology ring $H^*(K/T;\Z)$, we need further 
considerations.  In \cite{Toda1}, Toda established a method to describe the integral cohomology 
ring $H^*(K/T;\Z)$ by a minimal  system of generators and relations, from the mod $p$ cohomology rings 
$H^*(K;\Z/p\Z)$ and the rational cohomology ring $H^*(K/T;\Q)$.  In general, besides the two-dimensional 
generators,  there are extra generators of higher degrees,  and hence  the characteristic 
homomorphism $c$ is not surjective over $\Z$ in that case.   Along the lines of Toda's method, 
the integral cohomology rings 
of flag manifolds for all compact simply connected simple Lie groups have been computed (see  \cite{Bor1}, 
\cite{Bott-Sam1}, \cite{Toda-Wat1}, \cite{Nak1}, \cite{Nak2}). However, as mentioned in the 
introduction,  the generators have less geometric meaning in the Borel  presentation.

We pass to  reviewing   the  Schubert presentation. Recall   the  Bruhat decomposition, 
     \begin{equation*} 
            G = \coprod_{w \in W} B\dot{w}B,  
     \end{equation*} 
where $\dot{w}$ denotes any representative of $w$ in $W = N_{K}(T)/T$, $N_{K}(T)$ is the normalizer of 
$T$ in $K$.  It  induces a cell decomposition,  
   \begin{equation*} 
       G/B = \coprod_{w \in W} B\dot{w}B/B, 
  \end{equation*} 
where 
  \begin{equation*} 
        X_{w}^{\circ}  = B\dot{w}B/B \cong \C^{l(w)} 
  \end{equation*} 
is called the {\it Schubert cell}.  Note   that we have a homeomorphism $K/T \cong G/B$ by the Iwasawa 
decomposition.  The  {\it Schubert variety} $X_{w}$ is defined to be the closure of $X_{w}^{\circ}$. 
Then it is known that 
 \begin{equation*}
       X_{w} = \coprod_{v \leq w} X_{v}^{\circ},   
 \end{equation*} 
where $\leq$ is the Bruhat-Chevalley  ordering.  The fundamental class $[ X_{w} ]$ of $X_{w}$  lies in 
$H_{2l(w)}(G/B;\Z)$.  We define the cohomology class $Z_{w}  \in H^{2l(w)}(G/B;\Z)$ as the Poincar\'{e}
dual of $[X_{w_{0} w}] \in H_{2N - 2l(w)}(G/B;\Z)$, where $N$ is the complex dimension of the flag manifold
$G/B$. We call $Z_{w}$  the {\it Schubert class}. The Schubert classes 
$\{ Z_{w} \}_{w \in W}$ form an additive basis for $H^{*}(G/B;\Z)$.  We refer to $\{ Z_{w} \}_{w \in W}$ 
as the {\it Schubert basis}.   In order to complete the description of $H^*(G/B;\Z)$, we have to  compute 
the intersection multiplicities. Namely, given $u, v \in W$, we  can put 
\begin{equation*} 
       Z_{u} \cdot Z_{v} = \hspace{-0.8cm} 
                             \sum_{\tiny{ \begin{array}{ccc}
                                           & w \in W, \\ 
                                           &  l(u) + l(v) = l(w)
                                        \end{array} 
                                       }
                                  } \hspace{-0.8cm}  a_{u,v}^{w} Z_{w}  
\end{equation*} 
for some integers $a_{u,v}^{w}$, and we have to determine these ``structure
constants" $a_{u,v}^{w}$. As for this problem,  several results are available. 
For example, we have the following

\begin{thm}[Chevalley formula \cite{Che1}]  \label{thm:Chevalley} 
  If $w \in W, \; \alpha \in \Pi$,  then we have the  formula  
      \begin{equation*} 
              Z_{s_{\alpha}} \cdot Z_{w} = \hspace{-1cm} 
                                                   \sum_{\tiny{ \begin{array}{ccc} 
                                                                    & \beta \in \Delta^{+},  \\
                                                                    & l(w s_{\beta}) = l(w) + 1
                                                                 \end{array}
                                                                }
                                                         }  
                                       \hspace{-1cm}  (\beta^{\vee}|\omega_{\alpha}) Z_{w s_{\beta}}.   
      \end{equation*} 
\end{thm} 

For  recent developments in the Schubert calculus, in particular,   on  multiplying Schubert classes, 
see also \cite{Duan1}, \cite{Pra1}.

\section{Schubert calculus on flag manifolds}
As reviewed  in the previous section,  there are two different ways of describing  the 
integral cohomology ring of  $K/T$, and therefore we have two bases for $H^{*}(K/T;\Z)$
correspondingly.  One is the ``algebraic basis" derived from the Borel presentation and the other is 
the ``geometric basis" $\{ Z_{w} \}_{w \in W}$ consisting  of the Schubert classes.   
It is interesting  to know how these two bases are related.   More precisely, we wish  to 
express the ring generators obtained in the Borel presentation in terms of 
the Schubert basis.  Our main tool will be  the  divided difference operators  introduced 
independently by  Bernstein-Gelfand-Gelfand \cite{BGG1} and Demazure \cite{Dem1}. We now recall 
their definition.   For $\alpha \in \Delta$,  we define an  endomorphism of $H^*(BT;\Z)$ by 

  \begin{equation*} 
        \Delta_{\alpha} (u) = \dfrac{u - s_{\alpha}(u)}{\alpha}, \quad  u \in H^*(BT;\Z). 
  \end{equation*}

\begin{definition} 
  For $w \in W$,  we define the operator 
   \[   \Delta_{w} = \Delta_{\alpha_{i_{1}}} \circ \cdots \circ  \Delta_{\alpha_{i_{k}}}   \]
  on $H^*(BT;\Z)$ lowering the degree by $2 l(w)$, 
  where $w = s_{i_{1}} \cdots s_{i_{k}}$ is a reduced decomposition of $w$. 
\end{definition} 
One can show that the operator $\Delta_{w}$ is well defined, i.e., it is independent of the 
choice of the reduced decomposition of $w$. 

  Note that the divided difference operators $\Delta_{\alpha}, \alpha \in \Delta$,  are  characterized 
  by the following two properties:  
\begin{equation}   \label{eqn:BGG1}
    \Delta_{\alpha} (\omega_{\beta}) = \delta_{\alpha \beta},
\end{equation} 
\begin{equation} \label{eqn:BGG2}
           \Delta_{\alpha}(uv) = \Delta_{\alpha}(u) v + s_{\alpha}(u) \Delta_{\alpha}(v)  
 \end{equation}  
      for $u, v \in H^*(BT;\Z)$. 

The characteristic homomorphism 
   \[  c:  H^*(BT;\Z) \longrightarrow H^*(K/T;\Z)  \] 
can be   described by the divided difference operators.    Since $\{ Z_{w} \}_{w \in W}$
 is a $\Z$-basis  for $H^*(K/T;\Z)$,   we can put 
   \begin{equation*} 
         c(f) = \sum_{l(w) = k} a_{w} Z_{w}, \quad a_{w} \in \Z  
   \end{equation*} 
for a polynomial $f \in H^{2k}(BT;\Z)$. So we have to determine the coefficients 
$a_{w}$. This problem was solved independently by  Bernstein-Gelfand-Gelfand
\cite{BGG1} and Demazure \cite{Dem1}.

\begin{thm}[Bernstein-Gelfand-Gelfand \cite{BGG1}, Demazure \cite{Dem1}] \label{thm:char.hom}
 For a polynomial $f \in H^{2k}(BT;\Z)$, we have 
   \begin{equation*} 
         c(f) =  \sum_{l(w) = k} \Delta_{w}(f) Z_{w}. 
    \end{equation*} 
  In particular, for $\alpha \in \Pi$, we have 
     \begin{equation*}  
          c(\omega_{\alpha}) =  Z_{s_{\alpha}}. 
    \end{equation*} 
\end{thm} 

In addition to this, using the divided difference operators, we can express an arbitrary Schubert class $Z_{w}$ 
as a polynomial  in the variables $Z_{s_{i}}$.  This is the Giambelli formula which we now recall 
(for details, see \cite{Hil1}, Section 3).  Consider the element 
\begin{equation*} 
      d = \prod_{\alpha \in \Delta^{+}} \alpha.   
\end{equation*} 
Then we have 
\begin{thm}[Giambelli formula]  \label{thm:Giambelli}  
   The Schubert class $Z_{w}$ corresponding to $w \in W$ is given by    
    \[ Z_{w} = c \left ( \Delta_{w^{-1}w_{0}} \left ( \dfrac{d}{|W|} \right )\right ).  \]        
\end{thm}

In Sections 4, 5 and 6, we exploit the above theorems  to find  the correspondence 
between ``algebraic bases" and ``geometric bases" in the cases of $K = SO(n), G_{2}$ and $F_{4}$.

 \section{The cases of $B_{n}$ and $D_{n}$}
 In this section, we consider the special orthogonal group $SO(n)$.  First we deal with 
  the odd special orthogonal group  $SO(2n + 1)$. 
  Let $T^n$ be the  standard maximal torus of $SO(2n+1)$, 
  \[   T^n = \left \{ \left ( 
             \begin{array}{rrrrrrrrr}
               & \cos 2\pi t_{1}  & -\sin 2\pi t_{1} &   &   &  &    \\
               & \sin 2\pi t_{1}  & \cos 2\pi t_{1} &     \\ 
               &                  &                 & \ddots \\
               &                  &                 &     &   \cos 2\pi t_{n} & -\sin 2\pi t_{n} \\
               &                  &                 &     &   \sin 2\pi t_{n} & \cos 2\pi t_{n} \\
               &                  &                 &     &                   &                 & 1  
                   \end{array}    \right )   \right \}.     \]
Then we have an isomorphism:
   \[  H^*(BT^n;\Z) = \Z[t_{1}, t_{2}, \ldots, t_{n} ].  \] 
The system of  simple roots is given by 
   \[ \Pi = \{ \alpha_{1} = t_{1} - t_{2}, \alpha_{2} = t_{2} - t_{3}, \ldots, 
               \alpha_{n-1} = t_{n-1} - t_{n}, \alpha_{n} = t_{n}  \}.   \]
The corresponding fundamental weights $\{ \omega_{i} \}_{1 \leq i \leq n}$ are  given by 
  \begin{equation} \label{eqn:weights.SO(2n+1)}
     \begin{array}{llll} 
         \omega_{i} & = t_{1}  + t_{2} + \cdots + t_{i} \quad  (1 \leq i \leq n-1), \medskip \\
         \omega_{n} & = \dfrac{1}{2}(t_{1}  + \cdots + t_{n}).  \medskip 
     \end{array} 
  \end{equation} 
  Let $s_{i} \, (1 \leq i \leq n)$ be the reflection corresponding to the simple root 
  $\alpha_{i} \, (1 \leq i \leq n)$. Then the Weyl group $W(SO(2n+1))$ is  finite and  
  is generated by $s_{i} \, (1 \leq i \leq n)$ which act on $\{ t_{i} \}_{1 \leq i \leq n}$ 
  as permutations and  signs changes:
        \begin{equation} 
           \begin{array}{llll} 
                &  W(SO(2n+1)) & = \langle s_{1}, s_{2}, \ldots, s_{n} \rangle  \\
                &              & \cong \mathcal{S}_{n} \ltimes {(\Z/2\Z)^n}, 
           \end{array} 
         \end{equation} 
where $\mathcal{S}_{n}$ is the symmetric group of $n$ letters and $\ltimes$ means  the semi-direct 
product. 

Now we recall  the Borel presentation of $H^*(SO(2n+1)/T^n;\Z)$ that had been  probably  known, in some form, 
already  to  Borel.  However, in an  explicit form, it  first  appeared in \cite{Toda-Wat1} 
as far as the author knows.  
  \begin{thm}[Toda-Watanabe \cite{Toda-Wat1}, Theorem 2.1]  \label{thm:SO(2n+1)/T^n}
    The integral cohomology ring of $SO(2n+1)/T^n$ is 
      \begin{align*} 
          H^*(SO(2n+1)/T^n;\Z) & = \Z[t_{1}, t_{2}, \ldots, t_{n}, \gamma_{1}, \gamma_{2}, 
                                 \ldots, \gamma_{n}]  \\
                             & \left /\left (  \begin{array}{llll} 
                                             & c_{i} - 2\gamma_{i} \; (1 \leq i \leq n), \\
                                             & \gamma_{2k} + \displaystyle{\sum_{i=1}^{2k-1} 
                                               (-1)^{i} \gamma_{i} \gamma_{2k-i} }
                                               \; (1 \leq k \leq n)  
                                           \end{array}  \right ),    \right.                                          
      \end{align*}      
where we denote by the same symbols $t_{i} \in H^2(SO(2n+1)/T^n;\Z)$ the images of 
$t_{i} \in H^2(BT^n;\Z)$ under the homomorphism  $c$,  $c_{i} = e_{i}(t_{1}, \ldots, t_{n}) \; 
 (1 \leq i \leq n)$, and   $\gamma_{i} = 0$ for  $i > n$. 
  \end{thm} 

We wish  to express the algebra generators $\{ t_{1}, \ldots, t_{n}, \gamma_{1}, \ldots, 
\gamma_{n} \}$ in terms of Schubert classes. For simplicity, we denote the Schubert class 
corresponding to the element   $w = s_{i_{1}}s_{i_{2}}  \cdots s_{i_{k}}$ 
by $Z_{i_{1}i_{2}\cdots i_{k}}$, although the reduced decomposition of a Weyl group element may not
be unique.    The  correspondence between elements of degree 2 is easy. 
By (\ref{eqn:weights.SO(2n+1)}), we have   
  \begin{equation}  \label{eqn:t_i.SO(2n+1)}
     \begin{array}{lll} 
        t_{1} &= \omega_{1}, \medskip \\
        t_{i} & = -\omega_{i-1} + \omega_{i} \quad  (2 \leq i \leq n-1), \medskip \\
        t_{n} & = -\omega_{n-1} + 2\omega_{n}. 
     \end{array} 
  \end{equation} 
Since $c(\omega_{i}) = Z_{i}$, it follows from (\ref{eqn:t_i.SO(2n+1)}) that 
  \begin{equation}  \label{eqn:basis.SO(2n+1)}  
    \begin{array}{lll} 
      t_{1} &= Z_{1}, \medskip \\
      t_{i} & = -Z_{i-1} + Z_{i} \quad  (2 \leq i \leq n-1), \medskip \\
      t_{n} & = -Z_{n-1} + 2Z_{n}.
     \end{array} 
  \end{equation}

Next, for $\gamma_{k} \; (1 \leq k \leq n)$,  we can put 
       \[   \gamma_{k} = \sum_{l(w) = k} a_{w} Z_{w}  \]
for some integers $a_{w}$.  We  need   to determine the coefficients  $a_{w}$.  To  this end,  
we make use of the divided difference operators recalled  in the previous section. In this case, 
the characteristic homomorphism $c$ is not surjective over $\Z$ and $\gamma_{k} \; (1 \leq k \leq n)$ 
is not contained in the image  of $c$. (Strictly speaking, we should consider the spinor group, 
because  the special orthogonal group  is not simply connected.  In that case, $\gamma_{1}$ is contained 
in the image of $c$.)  However,  $2\gamma_{k} = c_{k}$ is contained in the image of $c$. So we can apply 
Theorem \ref{thm:char.hom} to the polynomial $c_{k} = e_{k}(t_{1}, \ldots, t_{n})$. Thus   
    \[  c_{k} = \sum_{l(w) = k} \Delta_{w}(c_{k}) Z_{w}.   \]
Let us compute $\Delta_{w}(c_{k})$,  where $l(w) = k$ for fixed $k  \; (1 \leq k \leq n)$.
For convenience of computation, we introduce the notation 
       \[  c_{l}^{(m)} = e_{l}(t_{1}, t_{2}, \ldots, t_{m}) \quad (1 \leq m \leq n, \; 1 \leq l \leq m),    \]
so that $c_{k} = c_{k}^{(n)}$. 

We need the following auxiliary result. 
\begin{lem}   \label{lem:delta(c_k).odd}
  For fixed $k, \; 1 \leq k \leq n$,   we have 
  \begin{enumerate} 
   \item    $\Delta_{i}(c_{k}^{(n)}) = 0 \quad  (1 \leq i \leq n-1)$, 
   \item    $ \Delta_{n}(c_{k}^{(n)}) = 2c_{k-1}^{(n-1)}$,   
    \item   $\Delta_{i}(c_{k-j}^{(n-j)})= 0 \quad (1 \leq i \leq n - j - 1, \; 1 \leq j \leq n - 1)$,  
    \item   $\Delta_{n-j}(c_{k-j}^{(n-j)}) =  c_{k-j-1}^{(n-j-1)}  \quad (1 \leq j \leq n-1)$. 
  \end{enumerate} 
\end{lem}

 \begin{proof} 
  (1) and (3) follow from the definition of $\Delta_{i}$ and the fact that 
        $c_{k}^{(n)}     = e_{k}(t_{1}, \ldots, t_{n})$ 
 (resp. $c_{k-j}^{(n-j)} = e_{k-j}(t_{1}, \ldots, t_{n-j})$) is invariant  under the action of 
  $s_{i} \; (1 \leq i \leq n-1)$ (resp. $s_{i} \; (1 \leq i \leq n-j-1)$). 

By (\ref{eqn:t_i.SO(2n+1)}), we have,  for $1 \leq i \leq n-1$,    
   \begin{equation}  \label{eqn:delta_i(t_j).SO(2n+1)} 
            \Delta_{i}(t_{j}) = \left \{ \begin{array}{rlll} 
                                          \hspace{-0.8cm}     1   & (j = i), \\
                                                             -1   & (j= i+1), \\
                                                              0   & (j \neq i, i+1)
                                          \end{array}    \right. 
      \end{equation} 
and      
    \begin{equation}  \label{eqn:delta_n(t_j).SO(2n+1)} 
             \Delta_{n}(t_{j})  = \left \{ \begin{array}{rrll} 
                                                    2   & (j = n), \\
                                                    0   &  (j \neq n). 
                                           \end{array}    \right.
        \end{equation}  
Then we compute 
  \begin{align*} 
      \sum_{k=0}^{n}\Delta_{n}(c_{k}^{(n)})
               &= \Delta_{n} \left (\prod_{i=1}^n (1 + t_{i}) \right ) \\
               &= \Delta_{n} \left (\prod_{i=1}^{n -1} (1 + t_{i}) \right )(1 + t_{n}) 
                + s_{n} \left (\prod_{i=1}^{n-1} (1 + t_{i})  \right ) \Delta_{n}(1 + t_{n}) \\
               &=  2 \prod_{i=1}^{n-1} (1 + t_{i}) \\
               &= 2\sum_{k=0}^{n-1}c_{k}^{(n-1)}
 \end{align*}
by (\ref{eqn:BGG2}) and (\ref{eqn:delta_n(t_j).SO(2n+1)}).  From this,   (2) follows.  (4) follows from 
a similar computation. 
\end{proof} 

By this lemma,  we deduce that 
  \begin{equation*} 
         \Delta_{w}(c_{k}) =  \left \{ \hspace{-0.3cm} 
                                       \begin{array}{lll} 
                                         & 2  & \text{if}  \quad w = s_{n-k+1} \cdots s_{n-1}s_{n}, \\
                                         & 0  & \text{otherwise}. 
                                       \end{array}  \right. 
  \end{equation*}
Therefore,  for $1 \leq k \leq n$, we have 
       \begin{equation*} 
             c_{k} = 2Z_{n-k+1, \ldots, n-1, n}
       \end{equation*} 
    in $H^*(SO(2n+1)/T^n; \Z)$. Since  $\gamma_{k}$  is defined by 
       \[    c_{k} = 2\gamma_{k}  \]
and $H^*(SO(2n+1)/T^n;\Z)$ is torsion free, we see that 
       \[   \gamma_{k} = Z_{n-k+1, \ldots, n-1, n}.   \]
Consequently,  we obtain the following result. 
\begin{prop}  \label{prop:1st_main1} 
  In Theorem $\ref{thm:SO(2n+1)/T^n}$, the relation between the ring  generators and the Schubert
  classes is given by
    \begin{align*}
          t_{1}      &= Z_{1}, \\
          t_{i}      & = -Z_{i-1} + Z_{i} \quad  (2 \leq i \leq n-1),  \\
          t_{n}      & = -Z_{n-1} + 2Z_{n}, \\
          \gamma_{k} &=  Z_{n-k+1, \ldots, n-1, n} \quad (1 \leq k \leq n).    
    \end{align*} 
In particular, we can take $Z_{1}, Z_{2}, \ldots, Z_{n}, Z_{n-1, n}, \ldots, Z_{12\cdots n-1, n}$
as the ring  generators of $H^*(SO(2n+1)/T^n;\Z)$. 
\end{prop}

\begin{rem} 
The standard projection $p$ from $SO(2n+1)/T^n$ to \\$SO(2n+1)/U(n)$ induces an injection 
      \[  p^*: H^*(SO(2n+1)/U(n);\Z) \hookrightarrow H^*(SO(2n+1)/T^n;\Z).  \]
The cohomology ring $H^*(SO(2n+1)/U(n);\Z)$ of the {\it odd orthogonal Grassmannian}
$SO(2n+1)/U(n)$ has a $\Z$-basis of Schubert classes $\{ \sigma_{\lambda} \}$ indexed 
by strict partitions $\lambda$ contained in the ``staircase'' $\rho_{n} = (n, n-1, 
\ldots, 1)$. Observe that  the generators  $Z_{n-k+1, \ldots, n-1, n} \; (1 \leq k \leq n)$  in 
Proposition $\ref{prop:1st_main1}$ are the $p^*$-images of the ``special Schubert classes"
$\sigma_{(k)} \; (1 \leq k \leq n)$ that were used by P. Pragacz to describe the cohomology 
ring of $SO(2n+1)/U(n)$ $($see \cite{Pra2}, Theorem $6.17)$.
\end{rem}

 The case of   the even special orthogonal group $SO(2n)$ is almost identical to  that of $SO(2n+1)$. 
 So we only exhibit the data and  results. Let  
  \[   T^n = \left \{ \left ( 
             \begin{array}{rrrrrrr}
               & \cos 2\pi t_{1}  & -\sin 2\pi t_{1} & & &   \\
               & \sin 2\pi t_{1}  & \cos 2\pi t_{1} &     \\ 
               &                  &                 & \ddots \\
               &                  &                 &     &   \cos 2\pi t_{n} & -\sin 2\pi t_{n} \\
               &                  &                 &     &   \sin 2\pi t_{n} & \cos 2\pi t_{n} \\ 
                   \end{array} \right )     \right \}   \]
be the standard maximal torus of $SO(2n)$.  Then we have an  isomorphism: 
   \[  H^*(BT^n;\Z) = \Z[t_{1}, t_{2}, \ldots, t_{n} ].  \] 
The system of  simple roots is given by 
   \[ \Pi = \{ \alpha_{1} = t_{1} - t_{2}, \alpha_{2} = t_{2} - t_{3}, \ldots, 
               \alpha_{n-1} = t_{n-1} - t_{n}, \alpha_{n} = t_{n-1} + t_{n} \}.   \]
The corresponding fundamental weights $\{ \omega_{i} \}_{1 \leq i \leq n}$ are  given by 
  \begin{equation} \label{eqn:weights.SO(2n)}
     \begin{array}{llll} 
         \omega_{i}   & = t_{1} + t_{2} + \cdots + t_{i} \quad  (1 \leq i \leq n-2), \medskip  \\
         \omega_{n-1} & = \dfrac{1}{2}(t_{1} + \cdots + t_{n-1} - t_{n}),  \medskip \\  
         \omega_{n}   & = \dfrac{1}{2}(t_{1} + \cdots + t_{n-1} + t_{n}).  \medskip  
     \end{array} 
  \end{equation} 
  Let $s_{i} \, (1 \leq i \leq n)$ be the reflection corresponding to the simple root 
  $\alpha_{i} \, (1 \leq i \leq n)$. Then the Weyl group $W(SO(2n))$ is  finite and  
  is generated by $s_{i} \, (1 \leq i \leq n)$ which act on $\{ t_{i} \}_{1 \leq i \leq n}$ 
  as permutations and an even number of  sign changes:
        \begin{equation} 
           \begin{array}{llll} 
                &  W(SO(2n)) & = \langle s_{1}, s_{2}, \ldots, s_{n} \rangle  \\
                &              & \cong \mathcal{S}_{n} \ltimes (\Z/2\Z)^{n-1}. 
           \end{array} 
         \end{equation} 

The Borel presentation of $H^*(SO(2n)/T^n;\Z)$ is given by 
  \begin{thm}[Toda-Watanabe \cite{Toda-Wat1}, Corollary 2.2]  \label{thm:SO(2n)/T^n}
   The integral cohomology ring of $SO(2n)/T^n$ is 
      \begin{align*} 
          H^*(SO(2n)/T^n;\Z) & = \Z[t_{1}, t_{2}, \ldots, t_{n}, \gamma_{1}, \gamma_{2}, 
                                 \ldots, \gamma_{n-1}]  \\
                             & \left /\left (  \begin{array}{llll} 
                                             & c_{i} - 2\gamma_{i} \; (1 \leq i \leq n-1), \;  c_{n}\\
                                             & \gamma_{2k} + \displaystyle{\sum_{i=1}^{2k-1} 
                                               (-1)^{i} \gamma_{i} \gamma_{2k-i} }
                                               \; (1 \leq k \leq n -1 )  
                                           \end{array}  \right ),    \right.                                          
      \end{align*}      
where we denote by the same symbols $t_{i} \in H^2(SO(2n)/T^n;\Z)$ the images of 
$t_{i} \in H^2(BT^n;\Z)$ under the homomorphism $c$,   $c_{i} = e_{i}(t_{1}, \ldots, t_{n}) \; (1 \leq i \leq n)$, 
and  $\gamma_{i} = 0$ for  $i \geq  n$. 
  \end{thm} 
By (\ref{eqn:weights.SO(2n)}), we have 
  \begin{equation} \label{eqn:t_i.SO(2n)}
    \begin{array}{lllll} 
        t_{1}   &= \omega_{1},  \medskip \\
        t_{i}   &= -\omega_{i-1} + \omega_{i} \quad  (2 \leq i \leq n-2),  \medskip \\
        t_{n-1} &= -\omega_{n-2} + \omega_{n-1} + \omega_{n},  \medskip \\
        t_{n}   &= -\omega_{n-1} + \omega_{n}. \medskip  
   \end{array} 
  \end{equation}  
Since $c(\omega_{i}) = Z_{i}$, it follows from (\ref{eqn:t_i.SO(2n)}) that 
\begin{equation} \label{eqn:basis.SO(2n)}
    \begin{array}{lllll} 
        t_{1}   &= Z_{1},  \medskip \\
        t_{i}   &= -Z_{i-1} + Z_{i} \quad  (2 \leq i \leq n-2),  \medskip \\
        t_{n-1} &= -Z_{n-2} + Z_{n-1} + Z_{n},  \medskip \\
        t_{n}   &= -Z_{n-1} + Z_{n}. \medskip  
   \end{array} 
  \end{equation}  

By (\ref{eqn:t_i.SO(2n)}), we have,  for $1 \leq i \leq n-1$,  
  \begin{equation}  \label{eqn:delta_i(t_j).SO(2n)} 
            \Delta_{i}(t_{j}) = \left \{ \begin{array}{rlll} 
                                          \hspace{-0.8cm}     1   & (j = i), \\
                                                             -1   & (j= i+1), \\
                                                              0   & (j \neq i, i+1)
                                          \end{array}    \right. 
      \end{equation} 
and      
    \begin{equation}  \label{eqn:delta_n(t_j).SO(2n)} 
             \Delta_{n}(t_{j})  = \left \{ \begin{array}{llll} 
                                                    1   & (j = n-1), \\
                                                    1   & (j = n), \\
                                                    0   &  (j \neq n-1, n). 
                                           \end{array}    \right.
        \end{equation}  
Then we obtain the following quite similarly to Lemma \ref{lem:delta(c_k).odd}.  
\begin{lem} For fixed $k, \; 1 \leq k \leq n-1$, we have 
  \begin{enumerate} 
   \item    $\Delta_{i}(c_{k}^{(n)}) = 0 \quad  (1 \leq i \leq n-1)$,  
   \item    $ \Delta_{n}(c_{k}^{(n)}) = 2c_{k-1}^{(n-2)}$,  
    \item   $\Delta_{i}(c_{k-j+1}^{(n-j)})= 0 \quad (1 \leq i \leq n - j - 1, \; 2 \leq j \leq n - 1), $
    \item   $\Delta_{n-j}(c_{k-j+1}^{(n-j)}) =  c_{k-j}^{(n-j-1)}  \quad (2 \leq j \leq n-1)$. 
  \end{enumerate} 
\end{lem}

By this lemma,  we deduce that 
  \begin{equation*} 
    \begin{array}{llll}
       &  \Delta_{w}(c_{1}) = \left \{ \hspace{-0.3cm} 
                                       \begin{array}{llll} 
                                          & 2 & \text{if} \quad w = s_{n}, \\
                                          & 0 & \text{otherwise}; 
                                       \end{array} \right.   \medskip \\
      &   \Delta_{w}(c_{k}) =  \left \{ \hspace{-0.3cm} 
                                       \begin{array}{lll} 
                                           & 2  & \text{if}  \quad w = s_{n-k} \cdots s_{n-2}s_{n}, \\
                                           & 0  & \text{otherwise}
                                       \end{array}  \right. \medskip 
     \end{array} 
  \end{equation*}
for $2 \leq k \leq n-1$. 

Therefore  we have 
       \begin{equation*} 
           \begin{array}{llll} 
             & c_{1} = 2Z_{n}, \medskip \\
             & c_{k} = 2Z_{n-k, \ldots, n-2, n}  \quad (2 \leq k \leq n-1) \medskip 
           \end{array} 
       \end{equation*} 
    in $H^*(SO(2n)/T^n; \Z)$.   Since $\gamma_{k}$ is defined by  
       \[    c_{k} = 2\gamma_{k}  \]
and  $H^*(SO(2n)/T^n;\Z)$ is torsion free, we see that  
       \begin{equation*} 
          \begin{array}{llll} 
              & \gamma_{1} = Z_{n}, \medskip \\
              & \gamma_{k} = Z_{n-k, \ldots, n-2, n} \quad (2 \leq k \leq n-1).  \medskip 
          \end{array} 
       \end{equation*} 
Consequently,  we obtain  the following result. 
\begin{prop} \label{prop:1st_main2}
  In Theorem $\ref{thm:SO(2n)/T^n}$, the relation between the ring  generators and the Schubert
  classes is given by 
    \begin{align*}
          t_{1}   &= Z_{1},  \\
          t_{i}   &= -Z_{i-1} + Z_{i} \quad  (2 \leq i \leq n-2),   \\
          t_{n-1} &= -Z_{n-2} + Z_{n-1} + Z_{n},  \\
          t_{n}   &= -Z_{n-1} + Z_{n}, \\
       \gamma_{1} &= Z_{n}, \\
       \gamma_{k} &=  Z_{n-k, \ldots, n-2, n} \quad  (2 \leq k \leq n-1).    
    \end{align*} 
In particular, we can take $Z_{1}, Z_{2}, \ldots, Z_{n}, Z_{n-2, n}, \ldots, Z_{12\cdots n-2, n}$ 
as the ring  generators of $H^*(SO(2n)/T^n;\Z)$. 
\end{prop}

\begin{rem} 
The standard projection $p$ from $SO(2n)/T^n$ to $SO(2n)/U(n)$ induces an injection 
      \[  p^*: H^*(SO(2n)/U(n);\Z) \hookrightarrow H^*(SO(2n)/T^n;\Z).  \]
The cohomology ring $H^*(SO(2n)/U(n);\Z)$ of the {\it even  orthogonal Grassmannian}
$SO(2n)/U(n)$ has a $\Z$-basis of Schubert classes $\{ \sigma_{\lambda} \}$ indexed 
by strict partitions $\lambda$ contained in  $\rho_{n-1} = (n-1, 
\ldots, 1)$. Observe that  the generators  $Z_{n}, Z_{n-k, \ldots, n-2, n} \; (2 \leq k \leq n-1)$  in 
Proposition $\ref{prop:1st_main2}$ are the $p^*$-images of the special Schubert classes
$\sigma_{(k)} \; (1 \leq k \leq n-1)$ that were also used by P. Pragacz to describe the cohomology 
ring of $SO(2n)/U(n)$ $($see \cite{Pra2}, Theorem $6.17')$.
\end{rem}

  \section{The case of $G_{2}$}
    In this section, we concentrate on  the case of the exceptional Lie group $G_{2}$. 
    Let $T$ be a maximal torus  of $G_{2}$. Following  \cite{Bour1}, we take the system of
     simple roots $\Pi = \{\alpha_{1}, \alpha_{2} \}$  and the corresponding fundamental weights 
    $\{ \omega_{1}, \omega_{2} \}$. Then we can identify 
        \[   H^*(BT;\Z) = \Z [\omega_{1}, \omega_{2} ].   \]
    Let $s_{i} \, (i = 1, 2)$ be the  reflection corresponding to the  
    simple root $\alpha_{i} \, (i = 1, 2)$. Then the Weyl group $W(G_{2})$ of $G_{2}$
    is  finite and is  generated by $s_{i} \, (i = 1,2)$:  
      \begin{equation} \label{eqn:Weyl.G_2}
         \begin{array}{lllll} 
               &   W(G_{2})  = \langle s_{1}, s_{2} \rangle,  \medskip    \\
               & s_{1}^2   = s_{2}^2 = 1, \; (s_{1}s_{2})^6 = 1.   \medskip   
         \end{array} 
       \end{equation} 
 
Now  we review the Borel presentation of $H^*(G_{2}/T;\Z)$. We put 
  \begin{equation} \label{eqn:t_i.G_2}
   \begin{array}{lll}
       t_{1} &= -\omega_{1},  \medskip  \\
       t_{2} &= -\omega_{1} + \omega_{2},  \medskip \\
       t_{3} &= 2\omega_{1} - \omega_{2}, \medskip \\
       c_{i} &= e_{i}(t_{1}, t_{2}, t_{3}).  \medskip 
   \end{array}
  \end{equation} 
Then we can write 
     \[    H^{*}(BT;\Z) = \Z[t_{1},t_{2},t_{3}]/(c_{1}).   \]
The action of $W(G_{2})$ on $\{ t_{1}, t_{2}, t_{3} \}$ is given by T{\scriptsize ABLE} 1.  
\begin{table}  [h]   \label{table:action(G_2)}
  \begin{center}  
   \begin{tabular}{|c|c|c|c} 
   \noalign{\hrule height0.8pt}
   \hfil $ $ &  $s_{1}$   & $s_{2}$  \\
   \hline
   $t_{1}$  &  $-t_{2}$  &  $t_{1}$  \\      
   \hline 
   $t_{2}$  &  $-t_{1}$  &  $t_{3}$  \\  
   \hline 
   $t_{3}$  &  $-t_{3}$  &  $t_{2}$ \\     
     \noalign{\hrule height0.8pt}
   \end{tabular} 
   \end{center}
 \caption{}
\end{table} 

\begin{rem} 
  The elements $\{ t_{i} \}_{i = 1, 2, 3}$ are  derived from the natural 
  inclusion $T \subset SU(3) \subset G_{2}$. 
\end{rem} 

The integral cohomology ring of $G_{2}/T$ was first  determined by  Bott-Samelson \cite{Bott-Sam1},  
but we prefer to  use  the presentation due to  Toda-Watanabe.  
\begin{thm}[Bott-Samelson \cite{Bott-Sam1}, Toda-Watanabe \cite{Toda-Wat1}] \label{thm:G_2/T} 
   The integral cohomology ring of $G_{2}/T$ is 
       \[  H^{*}(G_{2}/T;\Z) = \Z[t_{1}, t_{2}, t_{3}, \gamma_{3} ] 
                                 /(\rho_{1}, \rho_{2}, \rho_{3}, \rho_{6}),   \]
where  $\rho_{1} = c_{1}, \; \rho_{2} = c_{2}, \; \rho_{3} = c_{3} - 2\gamma_{3}, \; 
         \rho_{6} = \gamma_{3}^2$,   
and we denote by the same symbols $t_i \in H^{2}(G_2/T;\Z)$ the images of $t_i \in H^2(BT;\Z)$ under the 
homomorphism $c$. 
\end{thm}

By (\ref{eqn:Weyl.G_2}),   the elements of the Weyl group $W(G_{2})$ are given by 
the following table. 
\vspace{0.3cm} 
\begin{center} 
\begin{tabular}[h]{|c|ll|}
   \noalign{\hrule height0.8pt}
   \hfil $l(w)$ &  Elements of $W(G_{2})$  &  \\
   \hline
         $0$    &   1   & \\
  \hline 
         $1$    &  $s_{1}$                     & $s_{2}$   \\
   \hline  
         $2$    &   $s_{1}s_{2}$               & $s_{2}s_{1}$  \\
   \hline 
         $3$    &   $s_{1}s_{2}s_{1}$          & $s_{2}s_{1}s_{2}$  \\
   \hline 
         $4$    &   $s_{1}s_{2}s_{1}s_{2}$      & $s_{2}s_{1}s_{2}s_{1}$  \\
   \hline 
         $5$    &   $s_{1}s_{2}s_{1}s_{2}s_{1}$ & $s_{2}s_{1}s_{2}s_{1}s_{2}$  \\
   \hline 
         $6$   &    $s_{1}s_{2}s_{1}s_{2}s_{1}s_{2}$  &  \\ 
   \noalign{\hrule height0.8pt}
   \end{tabular} 
\end{center} 
\vspace{0.3cm} 
Therefore the corresponding  Schubert  basis for  $H^*(G_2/T;\Z)$ is given as follows:
\vspace{0.3cm} 
\begin{center} 
   \begin{tabular}{|c|c|c|c|c|c|c|c|}
   \noalign{\hrule height0.8pt}
   \hfil deg &  $0$ & $2$ & $4$ & $6$ & $8$ & $10$ & $12$   \\
   \hline
     & $1$ &  $Z_{1}$ & $Z_{12}$    & $Z_{121}$ & $Z_{1212}$ & $Z_{12121}$ & $$  \\
   \hline  
    &      &  $Z_{2}$ & $Z_{21}$ & $Z_{212}$ & $Z_{2121}$  & $Z_{21212}$  & $Z_{121212}$    \\
   \noalign{\hrule height0.8pt}
   \end{tabular}
\end{center}  
\vspace{0.3cm} 
  Here we denote $Z_{s_{i}}$  simply  by  $Z_{i}$ and so on. We wish  to express the algebra generators 
  $\{ t_{1}, t_{2}, t_{3}, \gamma_{3} \}$ in terms of Schubert classes. 
  Since $c(\omega_{i}) = Z_{i} \; (i = 1, 2)$, it follows   from (\ref{eqn:t_i.G_2}) that  
\begin{equation} \label{eqn:basis.G_2} 
\begin{array}{llll}  
  t_{1} &=  -Z_1, \medskip \\
  t_{2} &= -Z_{1} + Z_{2}, \medskip \\
  t_{3} &= 2Z_{1} - Z_{2}. \medskip 
\end{array}
\end{equation} 

Next  we can put 
  \begin{equation*} 
      \gamma_{3} = a_{121}Z_{121} + a_{212}Z_{212}  
   \end{equation*} 
for some integers $a_{121}, a_{212}$ and we  need   to determine the coefficients $a_{121}, a_{212}$.  
The characteristic homomorphism $c$ is not surjective over $\Z$ and  $\gamma_{3}$ is not contained 
in the image of $c$, but  $2\gamma_{3} = c_{3}$  is in the image of $c$.  Thus putting  
          \[  c_{3} = b_{121} Z_{121} + b_{212} Z_{212}  \] 
for some integers $b_{121}, b_{212}$, we can compute the coefficients 
$b_{121}, b_{212}$ using the divided difference operators.    By (\ref{eqn:t_i.G_2}),  we have  
  \begin{equation*}
  \begin{array}{llll}  
     c_{3} &= t_{1}t_{2}t_{3} \medskip   \\
           &= 2\omega_{1}^3 -3\omega_{1}^2 \omega_{2} + \omega_{1}\omega_{2}^2.  
  \end{array}
  \end{equation*}   
Therefore  we derive 
  \begin{align*} 
   b_{121} &= \Delta_{1}\Delta_{2}\Delta_{1}(c_{3})  \\
           &= \Delta_{1}\Delta_{2}\Delta_{1}(2\omega_{1}^3 -3\omega_{1}^2 \omega_{2} 
              + \omega_{1}\omega_{2}^2)  \\
           &= -2, \\
   b_{212} &= \Delta_{2}\Delta_{1}\Delta_{2}(c_{3})  \\
           &= \Delta_{2}\Delta_{1}\Delta_{2}(2\omega_{1}^3 -3\omega_{1}^2 \omega_{2} 
              + \omega_{1}\omega_{2}^2)  \\
           &= 0
   \end{align*}
from  (\ref{eqn:BGG1}), (\ref{eqn:BGG2}) and T{\scriptsize ABLE} 1. 

Thus we have 
 \[  c_{3} = - 2Z_{121}   \]  
   in $H^{6}(G_{2}/T;\Z)$.  Since  $\gamma_{3}$ is defined by 
  \[   c_{3} = 2\gamma_{3}   \]
and  $H^*(G_2/T;\Z)$ is torsion free,  we see that 
        \[  \gamma_{3} = -Z_{121}.   \]    
Consequently,  we obtain the following result.  
\begin{prop} \label{prop:1st_main3} 
  In Theorem $\ref{thm:G_2/T}$, the relation between the ring  generators $\{ t_{1}, t_{2}, t_{3}, 
  \gamma_{3} \}$ and the Schubert classes  is given by 
     \begin{equation*} 
          \begin{array}{llll}  
              t_{1} &=  -Z_1, \medskip \\
              t_{2} &= -Z_{1} + Z_{2}, \medskip \\
              t_{3} &= 2Z_{1} - Z_{2}, \medskip \\
              \gamma_{3} & = -Z_{121}.  
          \end{array}
\end{equation*}  
In particular, we can take $Z_{1}, Z_{2}, Z_{121}$ as the ring  generators of $H^*(G_{2}/T;\Z)$. 
\end{prop}

\section{The case of $F_{4}$} 
  In this section, we deal with  the case of the exceptional Lie group $F_{4}$. 
  Let $T$ be a maximal torus of  $F_{4}$.  Following  \cite{Bour1}, we take the system of 
   simple roots  $\Pi =  \{ \alpha_{i} \}_{1 \leq i \leq 4}$ 
  and the corresponding fundamental weights $\{ \omega_{i} \}_{1 \leq i \leq 4}$. Then we 
  can identify 
           \[  H^{*}(BT;\Z) = \Z[\omega_{1},\omega_{2},\omega_{3},\omega_{4}].  \]  
  Let $s_{i} \, (1 \leq i \leq 4)$ be the  reflection corresponding to the 
  simple root $\alpha_{i} \, (1 \leq i \leq 4)$. Then the Weyl group $W(F_{4})$ of 
  $F_{4}$ is finite and is  generated by $s_{i} \, (1 \leq i \leq 4)$:
   \begin{equation}  \label{eqn:Weyl.F_4} 
    \begin{array}{llll} 
       &  W(F_{4})= \langle s_{1}, s_{2}, s_{3}, s_{4} \rangle, \medskip   \\
       &  s_{1}^2 = s_{2}^2 = s_{3}^2 = s_{4}^2 = 1,   \medskip  \\
       &  (s_{1}s_{2})^3 = (s_{3}s_{4})^3 = (s_{2}s_{3})^4 = 1, \medskip \\
       &  s_{1}s_{3} = s_{3}s_{1}, s_{1}s_{4} = s_{4}s_{1}, s_{2}s_{4} = s_{4}s_{2}.   
    \end{array} 
    \end{equation}  

Now  we review the Borel presentation of $H^*(F_{4}/T;\Z)$.  We put 
  \begin{equation}  \label{eqn:t_i.F_4}
   \begin{array}{llll} 
       t_{1} &= -\omega_{4},  \medskip \\
       t_{2} &= \omega_{1} - \omega_{4}, \medskip \\
       t_{3} &= -\omega_{1} + \omega_{2}  - \omega_{4}, \medskip  \\
       t_{4} &= -\omega_{2} + 2\omega_{3} - \omega_{4},  \medskip \\      
       c_{i} &= e_{i}(t_{1},\ldots,t_{4}), \medskip \\ 
        t    &= \dfrac{1}{2}c_{1} = \omega_{3} - 2\omega_{4}. \medskip   
   \end{array}
 \end{equation} 
Then we can write 
     \[    H^{*}(BT;\Z) = \Z[t_{1},t_{2},t_{3},t_{4},t]/(c_{1}-2t).   \]
The action of $W(F_{4})$ on $\{ t_{i} \}_{1 \leq i \leq 4}$ is given by  T{\scriptsize ABLE} 2, 
where  blanks indicate the trivial action.
\begin{table}  [h]
  \begin{center}  
   \begin{tabular}{|c|c|c|c|c|} 
   \noalign{\hrule height0.8pt}
   \hfil $$ &  $s_{1}$ & $s_{2}$ & $s_{3}$ & $s_{4}$   \\
   \hline
   $t_{1}$ &  $$ & $$ &  $$            & $t_{1} - t$   \\      
   \hline 
   $t_{2}$ &  $t_{3}$ &  $$  & $$      & $t_{2} -t$     \\  
   \hline 
   $t_{3}$ &  $t_{2}$ &  $t_{4}$  & $$ & $t_{3} -t$\\     
   \hline 
   $t_{4}$ & $$ & $t_{3}$ & $-t_{4}$   & $t_{4} -t$  \\
   \hline 
   $t$     &  $$ & $$ & $t - t_{4}$ & $-t$  \\
     \noalign{\hrule height0.8pt}
   \end{tabular} 
   \end{center}
 \caption{}
\end{table} 

  \begin{rem} 
    The elements $\{ t_{i}  \}_{1 \leq i \leq 4}$ and $t$ are  derived from 
     the natural inclusion $T \subset Spin(9) \subset F_{4}$.  
  \end{rem} 

The integral cohomology ring of $F_{4}/T$ was  determined by Toda-Watanabe \cite{Toda-Wat1}.  
\begin{thm}[Toda-Watanabe \cite{Toda-Wat1}, Theorem A]  \label{thm:F_4/T}
  The integral cohomology ring of $F_{4}/T$ is 
  \[  H^*(F_{4}/T;\Z) = \Z[t_{1}, t_{2}, t_{3}, t_{4}, t, \gamma_{3}, \gamma_{4}]
                           /(\rho_{1}, \rho_{2}, \rho_{3}, \rho_{4}, \rho_{6}, \rho_{8}, \rho_{12}),   \]
 where 
\begin{align*}
    \rho_{1}  &=  c_{1} - 2t, \\
    \rho_{2}  &=  c_{2} - 2t^2, \\
    \rho_{3}  &=  c_{3} - 2\gamma_{3}, \\
    \rho_{4}  &=  c_{4} - 4t\gamma_{3} + 8t^4 - 3\gamma_{4}, \\
    \rho_{6}  &= \gamma_{3}^2 - 3t^2\gamma_{4} - 4t^3\gamma_{3} + 8t^6, \\
    \rho_{8}  &= 3\gamma_{4}^2 + 6t\gamma_{3}\gamma_{4} - 3t^4\gamma_{4} - 13t^8, \\
    \rho_{12} &= \gamma_{4}^3  -6t^4\gamma_{4}^2 + 12 t^8\gamma_{4} - 8t^{12},  
\end{align*}
and we denote by the same symbols $t_i \in H^2(F_{4}/T;\Z)$ the images of  $t_{i} \in H^2(BT;\Z)$ 
under the homomorphism $c$. 
\end{thm} 

\begin{rem} 
    In \cite{Toda-Wat1}, Toda and  Watanabe described $H^*(F_{4}/T;\Z)$  using the inclusion 
    $H^*(F_{4}/Spin(9);\Z) \hookrightarrow H^*(F_{4}/T;\Z)$ and  the known structure of 
    $H^*(F_{4}/Spin(9);\Z)$. Theorem $\ref{thm:F_4/T}$ is a  rewritten form  of  their result 
    in terms  of the $W(F_{4})$-invariants. 
\end{rem} 

By (\ref{eqn:Weyl.F_4}), the elements of the Weyl group $W(F_{4})$ of  
length $\leq 4$ are given by the following table. 
\vspace{0.5cm} 

\begin{center} 
\begin{tabular}[h]{|c|lllll|}
   \noalign{\hrule height0.8pt}
   \hfil $l(w)$ & Elements  &  of  $W(F_4)$ &  & &    \\
   \hline
         $0$    &   1   & & & &         \\
  \hline 
         $1$    &  $s_{1}$   &  $s_{2}$  &  $s_{3}$ & $s_{4}$  &   \\
   \hline  
         $2$    &   $s_{1}s_{2}$ &   $s_{1}s_{3}$  & $s_{1}s_{4}$   & $s_{2}s_{1}$  
                & $s_{2}s_{3}$   \\
                &   $s_{2}s_{4}$   & $s_{3}s_{2}$   & $s_{3}s_{4}$ &   $s_{4}s_{3}$  & \\
   
    \hline  
         $3$    &   $s_{1}s_{2}s_{1}$  & $s_{1}s_{2}s_{3}$ & $s_{1}s_{2}s_{4}$ &  $s_{1}s_{3}s_{2}$ 
                &   $s_{1}s_{3}s_{4}$    \\
                &   $s_{1}s_{4}s_{3}$  & $s_{2}s_{1}s_{3}$ & $s_{2}s_{1}s_{4}$  & $s_{2}s_{3}s_{2}$
                &   $s_{2}s_{3}s_{4}$    \\
                &   $s_{2}s_{4}s_{3}$  & $s_{3}s_{2}s_{1}$  & $s_{3}s_{2}s_{3}$  & $s_{3}s_{2}s_{4}$ 
                &   $s_{3}s_{4}s_{3}$    \\ 
                &   $s_{4}s_{3}s_{2}$  & & &  & \\
   \hline 
         $4$    &   $s_{1}s_{2}s_{1}s_{3}$    & $s_{1}s_{2}s_{1}s_{4}$  
                &   $s_{1}s_{2}s_{3}s_{2}$    & $s_{1}s_{2}s_{3}s_{4}$  
                &   $s_{1}s_{2}s_{4}s_{3}$  \\
                &   $s_{1}s_{3}s_{2}s_{1}$    & $s_{1}s_{3}s_{2}s_{3}$ 
                &   $s_{1}s_{3}s_{2}s_{4}$    & $s_{1}s_{3}s_{4}s_{3}$  
                &   $s_{1}s_{4}s_{3}s_{2}$   \\
                &   $s_{2}s_{1}s_{3}s_{2}$    & $s_{2}s_{1}s_{3}s_{4}$  
                &   $s_{2}s_{1}s_{4}s_{3}$    & $s_{2}s_{3}s_{2}s_{1}$ 
                &   $s_{2}s_{3}s_{2}s_{3}$   \\
                &   $s_{2}s_{3}s_{2}s_{4}$    & $s_{2}s_{3}s_{4}s_{3}$  
                &   $s_{2}s_{4}s_{3}s_{2}$    & $s_{3}s_{2}s_{1}s_{3}$ 
                &   $s_{3}s_{2}s_{1}s_{4}$   \\
                &   $s_{3}s_{2}s_{3}s_{4}$    & $s_{3}s_{2}s_{4}s_{3}$  
                &   $s_{3}s_{4}s_{3}s_{2}$    & $s_{4}s_{3}s_{2}s_{1}$ 
                &   $s_{4}s_{3}s_{2}s_{3}$  \\
     \noalign{\hrule height0.8pt}
   \end{tabular} 
\end{center} 
\vspace{0.5cm} 

We have the corresponding Schubert  basis $\{ Z_{w} \}_{w \in W(F_{4})}$. As before,   we denote  
$Z_{s_{i}}$  simply by $Z_{i}$ and so on.  We wish  to express the algebra generators 
$\{ t_{1}, t_{2}, t_{3}, t_{4}, t, \gamma_{3}, \gamma_{4} \}$ in terms of Schubert classes.  
Since $c(\omega_{i}) = Z_{i}$, it follows from  (\ref{eqn:t_i.F_4}) that  
 
  \begin{equation} \label{eqn:basis.F_4} 
    \begin{array}{llll} 
       t_{1} &= -Z_{4},  \medskip \\
       t_{2} &= Z_{1} - Z_{4}, \medskip \\
       t_{3} &= -Z_{1} + Z_{2}  - Z_{4}, \medskip  \\
       t_{4} &= -Z_{2} + 2Z_{3} - Z_{4},  \medskip \\       
        t    &=  Z_{3} - 2Z_{4}. \medskip   
   \end{array}
 \end{equation}

Next  we can put
   \[  \gamma_{3} = \sum_{l(w) = 3} a_{w} Z_{w},  \quad    \gamma_{4} = \sum_{l(w) = 4} a_{w} Z_{w}  \] 
for some integers $a_{w}$. We wish  to determine the coefficients $a_{w}$.  
By Theorem \ref{thm:F_4/T}, we have  
   \begin{equation} \label{eqn:gamma3,4}
       \begin{array}{llll} 
              2\gamma_{3} & = c_{3}, \medskip  \\
              3\gamma_{4} & = c_{4} - 4t\gamma_{3} + 8t^4  \medskip  \\
                           & = c_{4} - 2tc_{3}      + 8t^4.  \medskip 
       \end{array} 
    \end{equation}  
Therefore $2\gamma_{3}$ and $3\gamma_{4}$ are contained in the image of $c$.
So  as in the case of $G_{2}$,  we apply the divided difference operators to 
the right  hand side of (\ref{eqn:gamma3,4}).   The result can be seen in the following table.

\vspace{0.3cm} 
\begin{center} 
   \begin{tabular}{|c|c|c|c|c|c|c|c|}
   \noalign{\hrule height0.8pt}
   \hfil  &  $c_{3}$  \\
   \hline
     $\Delta_{121}$   & $0$    \\
   \hline  
     $\Delta_{123}$   & $2$    \\
   \hline
     $\Delta_{124}$   & $0$    \\
   \hline 
     $\Delta_{132}$   & $0$    \\
    \hline 
     $\Delta_{134}$   & $0$    \\
   \hline 
     $\Delta_{143}$   & $0$    \\
    \hline 
     $\Delta_{213}$   & $0$    \\
    \hline 
     $\Delta_{214}$   & $0$    \\
   \hline 
     $\Delta_{232}$   & $0$    \\
   \hline 
     $\Delta_{234}$   & $-2$   \\
   \hline 
     $\Delta_{243}$   & $-4$   \\
   \hline 
     $\Delta_{321}$   & $0$    \\
   \hline 
     $\Delta_{323}$   & $0$    \\
   \hline 
     $\Delta_{324}$   & $0$    \\
   \hline 
     $\Delta_{343}$   & $6$  \\
   \hline 
     $\Delta_{432}$   & $0$  \\
   \noalign{\hrule height0.8pt}
   \end{tabular}
\end{center}  
Here we denote $\Delta_{1}\Delta_{2}\Delta_{1}$ simply by $\Delta_{121}$ and so on.

Thus we have 
  \begin{equation*} 
  \begin{array}{llll} 
    c_{3} & = 2Z_{123} -2Z_{234} - 4Z_{243} + 6Z_{343} \medskip \\
          & = 2(Z_{123} - Z_{234} - 2Z_{243} + 3Z_{343}) \medskip 
  \end{array} 
  \end{equation*} 
in $H^{6}(F_{4}/T;\Z)$.  Since $\gamma_{3}$ is defined by
       \[  c_{3} = 2\gamma_{3}   \] 
and $H^{*}(F_{4}/T;\Z)$ is torsion free, we see that  
      \begin{equation}  \label{eqn:gamma_3(F_4)}
                 \gamma_{3} = Z_{123} - Z_{234} - 2Z_{243} + 3Z_{343}.   
      \end{equation}

Similarly  we obtain 
\vspace{0.3cm} 
\begin{center} 
   \begin{tabular}{|c|c|c|c|c|c|c|c|}
   \noalign{\hrule height0.8pt}
   \hfil  &  $c_{4} -2tc_{3}+ 8t^4$  \\
   \hline
     $\Delta_{1213}$  & $0$    \\
   \hline  
     $\Delta_{1214}$   & $0$    \\
   \hline
     $\Delta_{1232}$   & $0$    \\
   \hline 
     $\Delta_{1234}$   & $3$    \\
    \hline 
     $\Delta_{1243}$   & $-30$    \\
   \hline 
     $\Delta_{1321}$   & $0$    \\
    \hline 
     $\Delta_{1323}$   & $12$    \\
    \hline 
     $\Delta_{1324}$   & $0$    \\
   \hline 
     $\Delta_{1343}$   & $0$    \\
   \hline 
     $\Delta_{1432}$   & $0$   \\
   \hline 
     $\Delta_{2132}$   & $0$   \\
   \hline 
     $\Delta_{2134}$   & $0$    \\
   \hline 
     $\Delta_{2143}$   & $0$    \\
   \hline 
     $\Delta_{2321}$   & $0$    \\
   \hline 
     $\Delta_{2323}$   & $0$  \\
   \hline 
     $\Delta_{2324}$   & $0$  \\
   \hline 
     $\Delta_{2343}$   & $0$ \\
   \hline 
     $\Delta_{2432}$   & $0$  \\
   \hline 
     $\Delta_{3213}$   & $0$   \\
   \hline 
     $\Delta_{3214}$   & $0$  \\
   \hline 
     $\Delta_{3234}$   & $-3$  \\
  \hline
     $\Delta_{3243}$  &  $30$  \\
   \hline 
     $\Delta_{3432}$  &  $0$   \\
    \hline 
     $\Delta_{4321}$  &  $0$   \\
   \hline 
     $\Delta_{4323}$  &  $-24$   \\
   \noalign{\hrule height0.8pt}
   \end{tabular}
\end{center}

Thus we have 
  \begin{equation*} 
   \begin{array}{lllll} 
      c_{4} - 2tc_{3} + 8t^4 &= 
              3Z_{1234} - 30Z_{1243} + 12Z_{1323} - 3Z_{3234} + 30Z_{3243} - 24Z_{4323} 
              \medskip  \\
           & = 3(Z_{1234} - 10Z_{1243} + 4Z_{1323} - Z_{3234} + 10Z_{3243} - 8Z_{4323}) 
              \medskip   
  \end{array} 
  \end{equation*} 
in $H^{8}(F_{4}/T;\Z)$.  Since $\gamma_{4}$ is defined by 
     \[  c_{4} - 2tc_{3} + 8t^4 = 3\gamma_{4}  \]
and  $H^*(F_4/T;\Z)$ is torsion free,  we see that  
    \begin{equation}  \label{eqn:gamma_4(F_4)} 
        \gamma_{4} = Z_{1234} - 10Z_{1243} + 4Z_{1323} - Z_{3234} + 10Z_{3243} - 8Z_{4323}.
    \end{equation} 

Unfortunately, from  (\ref{eqn:gamma_3(F_4)}) and  (\ref{eqn:gamma_4(F_4)}), we cannot decide  
 which Schubert classes are  indecomposable. This leads us to the use of the Giambelli formula. 
 We need the following data. 
\begin{itemize} 
  \item [(1)]A set of positive roots of $F_{4}$ is given by 
           \[   \Delta^{+} = \left \{ \begin{array}{llll} 
                                      &  t_{i} \pm  t_{j}  \; (1 \leq i < j \leq 4),  \medskip \\  
                                      &  t_{i} \; (1 \leq i \leq 4), \; 
                                         \dfrac{1}{2}(t_{1} \pm t_{2} \pm t_{3} \pm t_{4})   \medskip   
                                 \end{array}   \right  \}.   \]
  \item [(2)]The longest element $w_{0}$ of the Weyl group $W(F_{4})$ is given  by 
            \[   w_{0} =    s_{1}s_{2}s_{1}s_{3}s_{2}s_{1}s_{3}s_{2}s_{3}s_{4}s_{3}s_{2}s_{1}s_{3}s_{2}s_{3}s_{4}s_{3}s_{2}s_{1}s_{3}s_{2}s_{3}s_{4}.   \]
\end{itemize} 
   
Then, using  the Giambelli formula (Theorem \ref{thm:Giambelli}), we obtain the following 
 result. 
  \begin{lem}  \label{lem:Schubert_polynom.F_4}
      In  $(\ref{eqn:gamma_3(F_4)})$ and $(\ref{eqn:gamma_4(F_4)})$, each Schubert class is 
       expressed in terms of $t_{i} \; (1 \leq i \leq 4), t, \gamma_{3}, \gamma_{4}$ 
       as follows$:$

   \begin{align*}
         Z_{123} &= \gamma_{3} - 2t_{1}^3 + 3t_{1}^2 t - 2t_{1}t^2, \\ 
         Z_{234} &= -t_{1}^3, \\
         Z_{243} &= -2t_{1}^3 + 3t_{1}^2 t - t_{1}t^2, \\
         Z_{343} &= -t_{1}^3 + t_{1}^2t,  \\
         Z_{1234} &= -\gamma_{4} + (t_{1} - 2t)\gamma_{3} + t_{1}^3t - t_{1}^2t^2 + 3t^4, \\
         Z_{1243} &= \gamma_{4} + (-2t_{1} + 2t)\gamma_{3} + 2t_{1}^4 - 4t_{1}^3 t + 3t_{1}^2 t^2 - 3t^4, \\
         Z_{1323} &= -\gamma_{4} - t\gamma_{3} + 2t_{1}^4 - 4t_{1}^3 t + 4t_{1}^2 t^2 - 2t_{1} t^3 + 3t^4, \\
         Z_{3234} &= \gamma_{4} + (-t_{1} + 2t)\gamma_{3} + t_{1}^4 - t_{1}^3 t + t_{1}^2 t^2 - 3t^4, \\
         Z_{3243} &= t_{1}^4 - 2t_{1}^3 t + t_{1}^2 t^2, \\
         Z_{4323} &= -\gamma_{4} + (2t_{1} - 2t)\gamma_{3} - t_{1}t^3 + 3t^4.
   \end{align*} 
      In particular, $Z_{123}$ and $Z_{1234}$ are indecomposable in   the ring $H^*(F_{4}/T;\Z)$. 
  \end{lem}

Consequently,  we obtain the following result. 
\begin{prop} \label{prop:1st_main4} 
In Theorem $\ref{thm:F_4/T}$, the relation between the ring  generators $\{ t_{1}, t_{2}, t_{3},
t_{4}, t,  \gamma_{3}, \gamma_{4}  \}$ and the Schubert classes  is given by 
     \begin{equation*} 
          \begin{array}{llll}   
       t_{1} &= -Z_{4},  \medskip \\
       t_{2} &= Z_{1} - Z_{4}, \medskip \\
       t_{3} &= -Z_{1} + Z_{2}  - Z_{4}, \medskip  \\
       t_{4} &= -Z_{2} + 2Z_{3} - Z_{4},  \medskip \\       
        t    &=  Z_{3} - 2Z_{4},  \medskip  \\
       \gamma_{3} & = Z_{123} - Z_{234} - 2Z_{243} + 3Z_{343}, \medskip \\ 
       \gamma_{4} & = Z_{1234} - 10Z_{1243} + 4Z_{1323} - Z_{3234} + 10Z_{3243} - 8Z_{4323}. \medskip 
          \end{array}
\end{equation*}  
Furthermore, we have 
     \begin{align*} 
         Z_{123} &= \gamma_{3} - 2t_{1}^3 + 3t_{1}^2 t - 2t_{1}t^2, \\
         Z_{1234} &= -\gamma_{4} + (t_{1} - 2t)\gamma_{3} + t_{1}^3t - t_{1}^2t^2 + 3t^4, 
      \end{align*} 
  and  we can take $Z_{1}, Z_{2}, Z_{3}, Z_{4}, Z_{123}, Z_{1234}$ as the  
  ring  generators of $H^*(F_{4}/T;\Z)$. 
\end{prop}

\section{The Chow rings of $\mathrm{SO}(n), \mathrm{Spin}(n), \mathrm{G}_{2}$ and $\mathrm{F}_{4}$}
In this section, using our description of  the integral cohomology rings of the  flag manifolds of types
$B_{n}, D_{n}, G_{2}$ and $F_{4}$ (see Sections 4,  5 and  6)  and a  remark of Grothendieck, we compute 
the Chow rings of the corresponding complex algebraic groups. 
  
As in Section 2, let $K$ be a compact simply connected simple Lie group,  $T$ its maximal torus,    
$G = K^{\C}$ the complexification of $K$ and  $B$  a Borel subgroup of $G$ containing $T$. 
Let $A^{i}(\cdot)$ denote the Chow group  of codimension $i$ cycles  up to rational equivalence
and $A( \cdot) = \bigoplus_{i \geq 0} A^{i}(\cdot)$. Given any character $\chi$ of $B$, that is, 
a homomorphism  of $B$ into $\C^{\times}$, we have the  associated  line bundle $L_{\chi}$ over $G/B$, 
which  defines  an element of $A^{1}(G/B)$, denoted 
by $c(\chi)$.  This induces a homomorphism  
  \[  c:  \hat{B} \longrightarrow   A^{1}(G/B),   \]
where $\hat{B}$ denotes the character group of $B$.  Extending this homomorphism $c$ by 
multiplicativity to  the symmetric algebra $S(\hat{B})$ of $\hat{B}$, one obtains  a  
homomorphism    
  \begin{equation}  \label{eqn:ch.hom2}
        c:  S(\hat{B}) \longrightarrow  A(G/B),   
  \end{equation} 
which is also called the {\it characteristic homomorphism}.

Then  Grothendieck's remark (\cite{Gro1}, p.21,  R{\scriptsize EMARQUES} $2^{\circ}$) allows us to obtain $A(G)$ as 
the quotient of $A(G/B)$ by   the ideal generated by $c(\hat{B})$. Denote by $T_{G}: A(G/B) \longrightarrow 
A(G)$ the canonical map onto the quotient. 
It is known (\cite{Gro1}, Lemme 10) that the Chow ring $A(G/B)$ of $G/B$ is isomorphic to the integral 
cohomology ring $H^*(G/B;\Z)$ of $G/B$. Under this isomorphism,  the Schubert variety $X_{w_{0}w}$ corresponds 
to the Schubert class $Z_{w}$  and the above characteristic 
homomorphism (\ref{eqn:ch.hom2}) coincides with  the characteristic homomorphism (\ref{eqn:ch.hom}).

Thus, in order to determine the Chow ring $A(G)$, we need  only to  compute the quotient ring of 
$H^*(K/T;\Z)$  by the ideal generated   by  $c(H^{2}(BT;\Z))$.  Since we assume that $K$ is simply 
connected,  $H^{2}(BT;\Z)$ is isomorphic to $H^{2}(K/T;\Z)$.  Therefore  we compute the   quotient ring  
$H^*(K/T;\Z)/(H^2(K/T;\Z))$.  We will  show  how to do this for  $F_{4}$. By Theorem 
\ref{thm:F_4/T},  we have 
     \begin{equation*} 
         \begin{array}{cccc} 
             H^*(F_4/T;\Z)/(t_{1}, t_{2}, t_{3}, t_{4}, t) 
                      & = \Z[\gamma_{3}, \gamma_{4}]/(2\gamma_{3}, 3\gamma_{4}, 
                             \gamma_{3}^2, \gamma_{4}^3). 
         \end{array} 
     \end{equation*} 
Taking  Proposition \ref{prop:1st_main4} into account, we can replace $\gamma_{3}, \gamma_{4}$ 
with $Z_{123}, Z_{1234}$ respectively.  Thus we obtain the following 

\begin{thm}  \label{thm:2nd.main4} 
   If $G$ is of type $\mathrm{F}_{4}$, we have
   \[   A(\mathrm{F}_{4})   = \Z[X_{3}, X_{4}]/(2X_{3}, 3X_{4}, X_{3}^2, X_{4}^3),  \]
    where  $X_{3}$ $($resp. $X_{4}$$)$ is the image under  $T_{G}$ of the element of $A(G/B)$ 
    defined by the Schubert variety $X_{w_{0}s_{1}s_{2}s_{3}}$ $($resp. $X_{w_{0}s_{1}s_{2}s_{3}s_{4}}$$)$. 
\end{thm}

In a similar way, we can compute the Chow rings of $\mathrm{SO}(n), \mathrm{Spin}(n)$ and $\mathrm{G}_{2}$. 
 The results are summarized  as follows. 

\begin{thm}  \label{thm:2nd.main1} 
    If $G$ is of type $B_{n} \, (\mathrm{SO}(2n+1)$ or $\mathrm{Spin}(2n+1))$ and $x_{i} \; (1 \leq i \leq n)$
    is the element of $A(G/B)$ defined by the Schubert variety $X_{w_{0}s_{n-i+1} \cdots s_{n-1}s_{n}}$. 
    Then we have 
  \begin{align*} 
      A(\mathrm{SO}(2n+1))   &= \Z[X_{1}, X_{3}, X_{5}, \ldots, X_{2[\frac{n+1}{2}]-1}]
                                     /(2X_{i}, \; X_{i}^{p_{i}}), \\
      A(\mathrm{Spin}(2n+1)) &= \Z[ X_{3}, X_{5}, \ldots, X_{2[\frac{n+1}{2}]-1}]
                                     /(2X_{i}, \; X_{i}^{p_{i}}),       
  \end{align*} 
 where  
           \[  p_{i} = 2^{[\log_{2}\frac{n}{i}] + 1},  \quad X_{i} = T_{G}(x_{i}).   \]
\end{thm}

\begin{thm}  \label{thm:2nd.main2} 
   If $G$ is of type $D_{n} \, (\mathrm{SO}(2n)$ or $\mathrm{Spin}(2n))$ and $x_{1}$ 
   $($resp. $x_{i} \, (2 \leq i \leq n-1))$ is the element of $A(G/B)$ defined by the Schubert 
   variety  $X_{w_{0}s_{n}}$ $($resp. $X_{w_{0}s_{n-i} \ldots s_{n-2}s_{n}} \, 
   (2 \leq i \leq n-1))$.  Then we have 
  \begin{align*} 
      A(\mathrm{SO}(2n))   &= \Z[X_{1}, X_{3}, X_{5}, \ldots, X_{2[\frac{n}{2}]-1}]
                                         /(2X_{i}, \; X_{i}^{p_{i}}), \\
      A(\mathrm{Spin}(2n)) &= \Z[ X_{3}, X_{5}, \ldots, X_{2[\frac{n}{2}]-1}]
                                         /(2X_{i}, \; X_{i}^{p_{i}}),       
  \end{align*} 
  where 
      \[  p_{i} = 2^{[\log_{2}\frac{n-1}{i}] + 1},  \quad X_{i} = T_{G}(x_{i}).   \] 
\end{thm} 

\begin{thm}  \label{thm:2nd.main3} 
   If $G$ is of type $\mathrm{G}_{2}$, we have 
   \[   A(\mathrm{G}_{2})   = \Z[X_{3}]/(2X_{3}, X_{3}^2),   \] 
    where $X_{3}$ is the image under $T_{G}$  of the element of $A(G/B)$ defined by the 
    Schubert variety  $X_{w_{0}s_{1}s_{2}s_{1}}$.  
\end{thm}

We observe that our results obtained in this section agree with those obtained in Marlin \cite{Mar1}.


\end{document}